\documentclass{article}

\usepackage{microtype}
\usepackage{graphicx}
\usepackage{subfigure}
\usepackage{amsmath,amssymb,amsthm}

\usepackage{booktabs} % for professional tables

\usepackage{hyperref}

\newtheorem{theorem}{Theorem}
\newtheorem{lemma}{Lemma}
\newtheorem{corollary}{Corollary}[theorem]
\usepackage[accepted]{icml2020}

\icmltitlerunning{Ill-Posedness and Optimization Geometry for Nonlinear Neural Network Training}

\begin{document}

\twocolumn[
\icmltitle{Ill-Posedness and Optimization Geometry for \\ Nonlinear Neural Network Training}



\icmlsetsymbol{equal}{*}

\begin{icmlauthorlist}
\icmlauthor{Thomas O'Leary-Roseberry}{oden}
\icmlauthor{Omar Ghattas}{oden}

\end{icmlauthorlist}

\icmlaffiliation{oden}{Oden Institute for Computational Engineering and Sciences, The University of Texas at Austin, Austin, TX}

\icmlcorrespondingauthor{Thomas O'Leary-Roseberry}{tom@oden.utexas.edu}

\icmlkeywords{Machine Learning, ICML}

\vskip 0.3in
]



\printAffiliationsAndNotice{}  % leave blank if no need to mention equal contribution

\begin{abstract}

In this work we analyze the role nonlinear activation functions play
at stationary points of dense neural network training problems. We
consider a generic least squares loss function training
formulation. We show that the nonlinear activation functions used in
the network construction play a critical role in classifying
stationary points of the loss landscape. We show that for shallow
dense networks, the nonlinear activation function determines the
Hessian nullspace in the vicinity of global minima (if they exist),
and therefore determines the ill-posedness of the training
problem. Furthermore, for shallow nonlinear networks we show that the
zeros of the activation function and its derivatives can lead to
spurious local minima, and discuss conditions for strict saddle
points. We extend these results to deep dense neural networks, showing
that the last activation function plays an important role in
classifying stationary points, due to how it shows up in the gradient
from the chain rule.

\end{abstract}

\section{Introduction}

% Questions:
% 
%
% \begin{itemize}
%
% \item Ill-posedness, what argument do we have to make about the
% issues arising from manifold solutions instead of unique solutions? 
% 
% \end{itemize}

Here, we characterize  the optimization geometry of nonlinear least-squares
regression problems for generic dense neural networks and analyze the
ill-posedness of the training problem. 
Neural networks are a popular nonlinear functional approximation
technique that are succesful in data driven approximation regimes. A
one-layer neural network can approximate any continuous function on a
compact set, to a desired accuracy given enough neurons
\cite{Cybenko1989,Hornik1989}. Dense neural networks have been shown
to be able to approximate polynomials arbitrarily well given enough
hidden layers \cite{SchwabZech2019}. While no general functional
analytic approximation theory exists for neural networks, they are
widely believed to have great approximation power for complicated
patterns in data \cite{PoggioLiao2018}.  

% Optimal parameters to reconstruct patterns in data are found
Training a neural network, i.e., determining optimal values of network
parameters to fit given data, can be accomplished by solving the
nonconvex optimization problem of minimizing a loss function (known as
empirical risk minimization). Finding a global minimum is NP-hard and
instead one usually settles for local minimizers
\cite{Bertsekas1997,MurtyKabadi1987}. Here we seek to characterize how
nonlinear activation functions affect the least-squares optimization
geometry at stationary points. In particular, we wish to characterize
the conditions for strict saddle points and spurious local
minima. Strict saddle points are stationary points where the Hessian
has at least one direction of strictly negative curvature. They do not
pose a significant problem for neural network training, since they can
be escaped efficiently with first and second order methods
\cite{DauphinPescanuGulcehre2014,JinChiGeEtAl2017,JinNetrapalliJordan2017,NesterovPolyak2006,
  OLearyRoseberryAlgerGhattas2019}. On the other hand, spurious local
minima (where the gradient vanishes but the data misfit is nonzero)
are more problematic; escaping from them in a systematic way may
require third order information \cite{AnandkumarGe2016}.

We also seek to analyze the rank deficiency of the Hessian of the loss
function at global minima (if they exist),
in order to characterize the ill-posedness of the nonlinear neural
network training problem.  Training a neural network is,
mathematically, an inverse problem; rank deficiency of the Hessian
often makes solution of the inverse problem unstable to perturbations
in the data and leads to severe numerical difficulties when using
finite precision arithmetic \cite{Hansen98}. While early termination
of optimization iterations often has a regularizing effect
\cite{Hanke95, EnglHankeNeubauer96}, and general-purpose
regularization operators (such as $\ell^2$ or $\ell^1$) can be invoked, when
to terminate the iterations and how to choose the regularization to
limit bias in the solution are omnipresent challenges. On the other
hand, characterizing the nullspace of the Hessian can provide a basis
for developing a principled regularization operator that
parsimoniously annihilates this nullspace, as has been recently done
for shallow linear neural networks \cite{ZhuSoudry2018}.

We consider both shallow and deep dense neural network
parametrizations. The dense parametrization is sufficiently general
since convolution operations can be represented as cyclic matrices with
repeating block structure. For the sake of brevity, we do not consider
affine transformations, but this work can easily be extended to this
setting. We begin by analyzing shallow dense nonlinear networks, for which we
show that the nonlinear activation function plays a critical role in
classifying stationary points. In particular, if the neural network
can exactly fit the data, and zero misfit global minima exist, we show
how the Hessian nullspace depends on the activation function and its
first derivative at these points. 

For linear networks, results about local minima, global minima,
strict saddle points, and optimal regularization operators have been shown
\cite{BaldiHornik1989,ZhuSoudry2018}. The linear network case is a nonlinear matrix factorization problem, given data matrices $X \in \mathbb{R}^{n \times d}, Y \in \mathbb{R}^{m \times d}$, one seeks to find $W_1^* \in \mathbb{R}^{m \times r},W_0^* \in \mathbb{R}^{r \times n}$ such that they minimize
\begin{equation}
	\frac{1}{2}\|Y - W_1W_0X\|_F^2.
\end{equation}
When the data matrix $X$ has full row rank, then by the Eckart-Young Theorem, the solution is given by the rank $r$ SVD of $YX^T(XX^T)^{-1}$, which we denote with a subscript $r$ 
\begin{equation}
	W_1^*W_0^* = [YX^T(XX^T)^{-1}]_r.
\end{equation}
The solution is non-unique since for any invertible matrix $B \in \mathbb{R}^{r\times r}$
\begin{equation} \label{linear_ill_posedness}
	(W_1^*B)(B^{-1}W_0^*)^= [YX^T(XX^T)^{-1}]_r
\end{equation}
is also a solution. We show that in addition to inheriting issues related to ill-posedness of matrix factorization, the nonlinear activation functions in the nonlinear training problem create ill-posedness and non-uniqueness.

We show that stationary points not corresponding to zero misfit global
minima are determined by the activation function and its first
derivative through an orthogonality condition. In
  contrast to linear networks, for which the existence of spurious local minima depends only on the rank of the training data and the weights, we show that for nonlinear networks, both spurious local
  minima and strict saddle points exist, and depend on the activation
  functions, the training data, and the weights.

We extend these results to deep dense neural networks where stationary
points can arise from exact reconstruction of the training data by the
network, or an orthogonality condition that involves the activation
functions of each layer of the network and their first derivatives.

For nonlinear neural networks, some work exists on analyzing networks
with ReLU activation functions; in particular Safran et.\ al.\ establish
conditions for the existence of spurious local minima for two layer
ReLU networks \cite{SafranShamir2017}.

\subsection{Notation and Definitions}
For a given matrix $A \in \mathbb{R}^{m \times n}$, its vectorization, $\text{vec}(A) \in \mathbb{R}^{mn}$ is an $mn$ vector that is the columns of $A$ stacked sequentially. Given a vector $z \in \mathbb{R}^m$, its diagonalization $\text{diag}(z) \in \mathbb{R}^{m\times m}$ is the diagonal matrix with entry $ii$ being component $i$ from $z$. The diagvec operation is the composition $\text{diagvec}(A) = \text{diag}(\text{vec})(A) \in \mathbb{R}^{mn\times mn}$, this is sometimes shortened to dvec. The identity matrix in $\mathbb{R}^{d \times d}$ is denoted $I_d$. 
We use the notation $\nabla_X f(X)$ to mean derivatives of a function $f$ with respect to a matrix $X$, and $\partial_{\text{vec}(X)}f(\text{vec}(X))$ when expressing derivatives with respect to a vectorized matrix $\text{vec}(X)$: $\partial_{\text{vec}(X)}f(X) = \frac{\partial f}{\partial\text{vec}(X)}$. For matrices $A\in \mathbb{R}^{m \times n}$ and $B \in \mathbb{p \times q}$, the Kronecker product $A \otimes B \in \mathbb{R}^{pm \times qn}$ is the block matrix
\begin{equation}
	A\otimes B = \left[\begin{array}{ccc} a_{11}B &\cdots& a_{1n}B \\
	\vdots  & \ddots & \vdots \\
	a_{m1}B & \cdots & a_{mn}B\end{array} \right]
\end{equation}
For matrices $A,B \in \mathbb{R}^{m \times n}$, $A\circ B \in \mathbb{R}^{m\times n}$ is the Hadamard (element-wise) product. For a matrix $A \in \mathbb{R}^{m \times n}$ and a matrix $B \in \mathbb{R}^{n \times k}$, the expression $A\perp B$ means that the rows of $A$ are orthogonal to the columns of $B$, and thus $AB=0$.

For a differentiable function $F:\mathbb{R}^{d_W} \rightarrow \mathbb{R}$, and a parameter $W_0\in \mathbb{R}^{d_W}$, we say that $W_0$ is a first order stationary point if $\nabla F(W_0)=0$, we say that $W_0$ is a strict saddle point if there exists a negative eigenvalue for the Hessian $\nabla^2 F$. We say that $W_0$ is a local minimum if the eigenvalues of the Hessian $\nabla^2 F$ are all nonnegative. We say that $W_0$ is a global minimum if $F(W_0) \leq F(W)$ for all $W \in \mathbb{R}^{d_W}$.

\section{Stationary Points of Shallow Dense Network} \label{shallow_section}
We start by considering a one layer dense neural network training problem. Given training data matrices $X \in \mathbb{R}^{n \times d}, Y \in \mathbb{R}^{m \times d}$, the shallow neural network architecture consists of an encoder weight matrix $W_0 \in \mathbb{R}^{r\times n}$, a nonlinear activation function $\sigma$ (which is applied element-wise), and then a decoder weight matrix $W_1\in \mathbb{R}^{m \times r}$. The training problem (empirical risk minimization) may then be stated as 
\begin{align}
	\min_{W_1,W_0} &F(W_1,W_0) = \sum_{i=1}^d \frac{1}{2}\|y_i - W_1\sigma(W_0x_i)\|_{\ell^2(\mathbb{R}^m)}^2 \nonumber \\
	 	&= \frac{1}{2}\|Y - W_1\sigma(W_0X)\|^2_{F(\mathbb{R}^{m \times d})}.
\end{align} 

We will begin by analyzing first order stationary points of the objective function $F$. 

\begin{theorem} \label{critical_points_thm}
The gradient of the objective function $F$ is given by
\begin{align}
	\nabla F(W_1,W_0) &= [\nabla_{W_1}F(D,E)^T, \nabla_{W_0}F(W_1,W_0)^T]^T \nonumber \\
	\nabla_{W_1}F(W_1,W_0) &= (W_1\sigma(W_0X) - Y)\sigma(W_0X)^T \\
	\nabla_{W_0}F(W_1,W_0) &=  \nonumber \\
	[\sigma'(W_0X)&\circ (W_1^T(W_1\sigma(W_0X) - Y))]X^T.
\end{align}
First order stationary points are characterized by two main conditions:
\begin{enumerate}
\item A global minimum where the misfit is exactly zero: $W_1\sigma(W_0X) = Y$. The possibility for which depends on the representation capability of the network, and the data. \label{global minimum}

\item A stationary point not corresponding to zero misfit: $\sigma'(W_0X) \circ W_1^T(W_1\sigma(W_0X) - Y) \perp X^T$, and $(W_1\sigma(W_0X) -Y) \perp \sigma(W_0X)^T$ \label{not_global_minimum}
\end{enumerate}
\end{theorem}

\begin{proof}
The partial derivatives of the objective function $F(W_1,W_0)$ are derived in Lemma \ref{derivation_of_grad_lemma}. At a first order stationary point of the objective function $F$ both partial derivatives must be zero:
\begin{align} 
	 (W_1\sigma(W_0X) - Y)\sigma(W_0X)^T  &= 0 \label{crit_pt_W1}\\
	 [\sigma'(W_0X)\circ (W_1^T(W_1\sigma(W_0X) - Y))]X^T &= 0 \label{crit_pt_W0}.
\end{align}
In the case that $W_1\sigma(W_0X) = Y$ then both terms are zero, and the corresponding choices of $W_1,W_0$ define a global minimum. This can be seen since $F$ is a nonnegative function, and in this case it is exactly zero. 

Stationary points where $W_1\sigma(W_0X) \neq Y$ are characterized by orthogonality conditions. If $\nabla_{W_1} F(W_1,W_0) =0$, then this means that $(W_1\sigma(W_0X) - Y) \perp \sigma(W_0X)^T$; that is, the rows of $W_1\sigma(W_0X) - Y$ and the columns of $\sigma(W_0X)^T$ are pairwise orthogonal. If $\nabla_{W_0} F(W_1,W_0)=0$, then this similarly means that $[\sigma'(W_0X)\circ (W_1^T(W_1\sigma(W_0X) - Y))] \perp X^T$
\end{proof} 

\begin{corollary} \label{crit_pt_corollary}
Any $W_1,W_0$ such that $\sigma(W_0X) = 0$ and $\sigma'(W_0X) \circ (W_1^T(W_1\sigma(W_0X) - Y)) = 0$ correspond to first order stationary points of the objective function $F$. In particular, any $W_0$ for which $\sigma(W_0X) = \sigma'(W_0X) = 0$ corresponds to a first order stationary point for all $W_1$.
\end{corollary}

This result implies that points in parameter space where the activation function and its derivatives are zero can lead to sub-optimal stationary points. Note that if a zero misfit minimum is not possible, there may or may not be an actual global minimum (there will always be a global infimum), but since the misfit is not zero any such point will still fall into the second category. In what follows we characterize the optimization geometry of the objective function $F$ at global minima, and degenerate points of the activation function, i.e. points for which $\sigma(W_0X) = \sigma'(W_0X) = 0$.

\subsection{Zero misfit minima}
Suppose that for given data $X,Y$, there exists $W_1,W_0,\sigma$ such that $W_1\sigma(W_0X) = Y$. As was discussed in Theorem \ref{critical_points_thm}, such points correspond to a global minimum. In what follows we characterize Hessian nullspace at these points, and corresponding ill-posedness of the training problem. 

\begin{theorem}{Characterization of Hessian nullspace at global minimum.}
Given data $X,Y,\sigma $ suppose there exist weight matrices $W_1,W_0$ such that $Y = W_1\sigma(W_0X)$. Suppose further that $W_1$ and $\sigma(W_0X)$ are full rank, then the Hessian nullspace is characterized by directions $\widehat{W_1},\widehat{W_0}$ such that 
\begin{align}
	\bigg[\widehat{W_1} \sigma(W_0X) + W_1(\sigma'(W_0X) \circ \widehat{W_0} X) \bigg]  \perp \sigma(W_0X)^T\\
	\bigg[[\sigma'(W_0X) \circ (W_1^T\widehat{W_1} \sigma(W_0X))] + \nonumber \\
	 [\sigma(W_0X) \circ (W_1^TW_1(\sigma'(W_0X) \circ \widehat{W_0} X))]\bigg] \perp X^T.
\end{align}

In particular for any direction $\widehat{W_0}$, such that the directional derivative $\sigma'(W_0X)\circ \widehat{W_0} X$ is zero, the weight matrices
\begin{align}
	&\widehat{W_0} \nonumber \\
	 &\widehat{W_1} =  -W_1(\sigma'(W_0X) \circ \widehat{W_0} X) \sigma(W_0X)^T[\sigma(W_0X)\sigma(W_0X)^T]^{-1}
\end{align}
are in the nullspace of the Hessian matrix $\nabla^2 F(W_1,W_0)$
\end{theorem}

\begin{proof}
Since the misfit is zero, the Hessian is exactly the Gauss-Newton Hessian, which is derived in Lemma \ref{gn_hessian_derivation}. The matrices $\widehat{W_0},\widehat{W_1}$ are in the nullspace of the Hessian if
\begin{equation}
	\bigg[\widehat{W_1} \sigma(W_0X) + W_1(\sigma'(W_0X) \circ \widehat{W_0} X) \bigg]\sigma(W_0X)^T = 0 
\end{equation}
\begin{align}
	\bigg[[\sigma'(W_0X) \circ (W_1^T\widehat{W_1} \sigma(W_0X))] + \nonumber \\
	 [\sigma(W_0X) \circ (W_1^TW_1(\sigma'(W_0X) \circ \widehat{W_0} X))]\bigg]X^T = 0 .
\end{align}
For this to be the case we need that $[\widehat{W_1} \sigma(W_0X) + W_1(\sigma'(W_0X) \circ \widehat{W_0} X)] \perp \sigma(W_0X)^T$ and $[[\sigma'(W_0X) \circ (W_1^T\widehat{W_1} \sigma(W_0X))] + [\sigma(W_0X) \circ (W_1^TW_1(\sigma'(W_0X) \circ \widehat{W_0} X))]] \perp X^T$. The Hessian nullspace is fully characterized by points $\widehat{W_1}, \widehat{W_0}$ that satisfy these two orthogonality constraints. One way in which these constraints are satisfied is if 
\begin{align}
  	\widehat{W_1} \sigma(W_0X) = - W_1(\sigma'(W_0X)\circ \widehat{W_0} X)  \label{w1_gn_null_eq}\\
  	[\sigma'(W_0X)\circ (W_1^T\widehat{W_1} \sigma(W_0X))] + \nonumber \\
  	[\sigma(W_0X) \circ (W_1^TW_1(\sigma'(W_0X) \circ \widehat{W_0} X))] = 0 \label{w0_gn_null_eq}.
\end{align}  
Subsituting \eqref{w1_gn_null_eq} into \eqref{w0_gn_null_eq} we have 
\begin{equation}
	[-\sigma'(W_0X) + \sigma(W_0X)]\circ (W_1^TW_1(\sigma'(W_0X) \circ \widehat{W_0} X)) = 0.
\end{equation}
The first term is nonzero if $\sigma \neq \exp$, since $\sigma(W_0X)$ is assumed to be full rank. For the Hadamard product to be zero, the second term must be zero:
\begin{equation}
	W_1^TW_1(\sigma'(W_0X) \circ \widehat{W_0} X) = 0.
\end{equation}
This is accomplished when $W_1^TW_1 \perp \sigma'(W_0X) \circ \widehat{W_0} X$. Since $W_1$ is full rank this condition reduces to $\sigma'(W_0X) \circ \widehat{W_0} X = 0$. Suppose that $\widehat{W_0}$ satisfies this directional derivative constraint, then we can find a corresponding $\widehat{W_1}$ such that $\widehat{W_1},\widehat{W_0}$ are in the Hessian nullspace from \eqref{w1_gn_null_eq}:
\begin{align}
	&\widehat{W_1} = \nonumber \\
	 - &W_1 (\sigma'(W_0X) \circ \widehat{W_0} X)\sigma(W_0X)^T[\sigma(W_0X)\sigma(W_0X)^T]^{-1}.
\end{align}
Note that $\sigma(W_0X)\sigma(W_0X)^T \in \mathbb{R}^{r\times r}$ is invertible since $\sigma(W_0X)$ is assumed to be full rank.

\end{proof}

This result shows that the Hessian may have a nontrivial nullspace at
zero misfit global minima; in particular, if there are any local
directions $\widehat{W_0}$ satisfying the directional derivative
constraint $\sigma'(W_0X)\circ \widehat{W_0} X = 0$, then the Hessian
is guaranteed to have at least one zero eigenvalue. If the Hessian has
at least one zero eigenvalue, then the candidate global minimum
$W_1,W_0$ is not unique, and instead is on a manifold of global
minima. Global minima are in this case weak minima.  

This result is similar to the non-uniqueness of the linear network training problem, Equation \eqref{linear_ill_posedness}. However in this case the linear rank constraints are obfuscated by the nonlinear activation function, and, additionally the zeros of the activation function lead to more possibility for Hessian rank-deficiency and associated ill-posedness.

For weak global minima, regularization schemes that annihilate 
% penalize directions
the Hessian nullspace while leaving the range space unscathed 
can be used to
% mitigate the ill-posedness of the problem.
make the training problem well-posed without biasing the
solution. 
Furthermore, such regularization
schemes
% that address the rank deficiency of the Hessian without polluting
% important directions
will accelerate the asymptotic convergence rates of second order
methods (Newton convergence deteriorates from quadratic to linear in
the presence of singular Hessians), thereby making them even more
attractive relative to first order methods. 

\subsection{Strict Saddle Points and Spurious Local Minima.}

As was shown in Theorem \ref{critical_points_thm} and Corollary
\ref{crit_pt_corollary}, there are stationary points where
the misfits are not zero. In this section we show that these points can be both strict saddle points as well as spurious local minima. 

Suppose the gradient is zero, but the misfit is nonzero. As was discussed in condition \ref{not_global_minimum} of Theorem \ref{critical_points_thm} such minima require orthogonality conditions for matrices that show up in the gradient. Corollary \ref{crit_pt_corollary} establishes that this result is achieved if $\sigma(W_0X) = \sigma'(W_0X) = 0$. Many activation functions such as ReLU, sigmoid, softmax, softplus, tanh have many points satisfying these conditions (or at least approximately satisfying these conditions, i.e. for small $\epsilon >0$, $\|\sigma'(W_0X)\|_F,\|\sigma(W_0X)\|_F \leq \epsilon$). Such stationary points are degenerate due to the activation functions. In what follows we show that while these points are likely to be strict saddles, it is possible that some of them have no directions of negative curvature and are thus spurious local minima. 

\begin{theorem}{Negative Curvature Directions at Degenerate Activation Stationary Points.}
Let $W_1$ be arbitrary and suppose that $W_0$ is such that $\sigma'(W_0X) = \sigma(W_0X) = 0$, negative curvature directions of the Hessian at such points are characterized by directions $\widehat{W_0}$ such that 
\begin{equation} \label{negative_curvature_w0hat}
	\sum_{k=1}^d \sum_{i=1}^r (\widehat{W_0}x^{(k)})_i^2 (\sigma''(W_0x^{(k)}))_i (W_1^Ty^{(k)})_i < 0.
\end{equation}
\end{theorem}

\begin{proof}
Since $\sigma'(W_0X)$ all of the terms in the Gauss-Newton Hessian are zero (see Lemma \ref{gn_hessian_derivation}). Further, all of the off-diagonal non Gauss-Newton portions are also zero. In this case the only block of the Hessian that is nonzero is the non Gauss-Newton $W_0-W_0$ block (see Lemma \ref{ngn_hessian_derivation}). We proceed by analyzing an un-normalized Rayleigh quotient for this block in an arbitrary direction $\widehat{W_0}$. From Equation \ref{w0w0_block_ngn} we can compute the quadratic form:
\begin{align}
	&\text{vec}(\widehat{W_0})^T(\partial_{\text{vec}(W_0)}\partial_{\text{vec}(W_0)}\text{misfit})^T\text{misfit}\text{vec}(\widehat{W_0}) \nonumber \\
	= & \text{vec}(\widehat{W_0}X)^T\text{dvec}((W_1^T(W_1\sigma(W_0X) - Y)) \nonumber \\
	& \qquad \qquad \circ \sigma''(W_0X)) \text{vec}(\widehat{W_0}X) \nonumber \\
	= & \text{vec}(\widehat{W_0}X)^T\text{vec}(\big[(W_1^T(W_1\sigma(W_0X) - Y)) \nonumber \\
	& \qquad \qquad \circ \sigma''(W_0X) \circ \widehat{W_0}X\big]) 
\end{align}
Expanding this term in a sum we have:
\begin{equation}
	\sum_{k=1}^d\sum_{i=1}^r (\widehat{W_0}x^{(k)})^2_i (\sigma''(W_0x^{(k)}))_i (W_1^T(W_1\sigma(W_0x^{(k)}) - y^{(k)})_i
\end{equation}
The result follows noting that $\sigma(W_0X) = 0$.
\end{proof}

Directions $\widehat{W_0}$ that satisfy the negative curvature condition \eqref{negative_curvature_w0hat} are difficult to understand in their generality, since they depend on $X,Y$ and $\sigma''$.  We discuss some example sufficient conditions.

\begin{corollary}{Saddle point with respect to one data pair.}
Given $W_1,W_0$, and a strictly convex activation function $\sigma$, suppose that $\sigma'(W_0X) = \sigma(W_0X) = 0$. Suppose that there is a data pair with $x^{(k)} \neq 0$ such that at least one negative component of $W_1^Ty^{(k)}$. Then $W_1,W_0$ is a strict saddle point.
\end{corollary}
\begin{proof}
If $(W_1^Ty^{(k)})_i <0$, and the $x^{(k)}_j \neq 0$, then the direction $\widehat{W_0}_{ij} = 1$ with all other components zero defines a direction of negative curvature.
\begin{align}
	\sum_{i=1}^r&(\widehat{W_0}x^{(k)})_i^2(\sigma''(W_0x^{(k)}))_i (W_1^Ty^{(k)})_i \nonumber \\
	= &(x^{(k)})_j^2(\sigma''(W_0x^{(k)}))_j(W_1^Ty^{(k)})_j <0.
\end{align}
\end{proof}

\begin{corollary}
Given $W_1,W_0$ and a strictly convex function $\sigma$ and all elements of one row of $W_1^TY$ are negative then $W_1,W_0$ is a strict saddle point.
\end{corollary}
\begin{proof}
Let the $i^{th}$ row of $W_1^TY$ satisfy this condition, then any choice of $\widehat{W_0} \neq 0$ such that all rows other than $i$ are zero will define a direction of negative curvature. 
\end{proof}

These conditions are rather restrictive, but demonstrate the nature of existence of negative curvature directions. As was stated before, the most general condition for a strict saddle is the existence of $\widehat{W_0}$ that satisfies Equation \eqref{negative_curvature_w0hat}. We conjecture that such an inequality shouldn't be hard to satisfy, but as it is a nonlinear inequality finding general conditions for the existence of such $\widehat{E}$ is difficult. We have the following result about how the zeroes of the activation function and its derivatives can lead to spurious local minima. 

\begin{corollary}
For a given $W_0$, if $\sigma(W_0X) = \sigma'(W_0X) = \sigma''(W_0X) = 0$, then the Hessian at this point is exactly zero and this point defines a spurious local minimum.
\end{corollary}

Such points exist for functions like ReLU, sigmoid, softmax, softplus etc. Any activation function that has large regions where it is zero (or near zero) will have such points. The question is then, how common are they? For the aforementioned functions, the function and its derivatives are zero or near zero when the argument of the function is sufficiently negative. For these functions, and a given tolerance $\epsilon>0$ there exists a constant $C\leq 0$ such that for all $\xi < C$, $\sigma(\xi) \leq \epsilon$, $\sigma'(\xi) \leq \epsilon$ and $\sigma''(\xi) \leq \epsilon$. For ReLU (which does not have any derivatives at zero) $C = \epsilon = 0$. In one dimension this condition is true for roughly half of the real number line for each of these functions. For the condition to be true for a vector it must be true elementwise. So for the condition
\begin{equation}
	\sigma(W_0x^{(k)}) \leq \epsilon \text{ and } \sigma'(W_0x^{(k)}) \leq \epsilon \text{ and }  \sigma''(W_0x^{(k)}) \leq \epsilon
\end{equation}
to hold for a given input datum $x^{(k)}$; the encoder array must map each component of $x^{(k)}$ into the strictly negative orthant of $\mathbb{R}^r$. The probability of drawing a mean zero Gaussian random vector in $\mathbb{R}^r$ that is in the strictly negative orthant is $2^{-r}$. Furthermore for this condition to hold for all of $W_0X$ means it must be true for each column of the matrix $W_0X$. The probability of drawing a mean zero Gaussian random matrix in $\mathbb{R}^{r \times d}$ such that each column resides in the strictly negative orthant is $2^{-rd}$. In practice the linearly encoded input data matrix $W_0X$ is unlikely to have the statistical properties of a mean zero Gaussian, but this heuristic demonstrates that these degenerate points may be improbable to encounter. If the Hessian is exactly zero, one needs third order information to move in a descent direction \cite{AnandkumarGe2016}.

\section{Extension to Deep Networks} \label{deep_section}

In this section we briefly discuss the general conditions for stationary points of a dense neural network. We consider the parameterization. In this case the weights for an $N$ layer network are $[W_0,W_1,\cdots,W_N]$, where $W_0 \in \mathbb{R}^{r_0 \times m}$, $W_N \in \mathbb{R}^{n \times r_{N-1}}$, and all other $W_j \in \mathbb{R}^{r_1 \times r_0}$. The activation functions $\sigma_j$ are arbitrary. The network parameterization is
\begin{equation} \label{deep_nn_form}
	W_N\sigma_N(W_{N-1}\sigma_{N_1}( \cdots \sigma_1(W_0X)\cdots )).
\end{equation}
We have the following general result about first order stationary points of deep neural networks.

\begin{theorem}{Stationary points of deep dense neural networks}
The blocks of the gradient of the least squares loss function for the deep neural network (Equation \eqref{deep_nn_form}) are as follows:
\begin{align}
\nabla_{W_j}F(\mathbf{W}) = \bigg[\sigma'_{j+1}(W_j\sigma_j \cdots \sigma_1(W_0X)\cdots) \circ \nonumber \\
								\big(W_{j+1}^T\big(\sigma_{j+2}'(W_{j+1}\sigma_j\cdots \sigma_1(W_0X)\cdots) \circ \nonumber \\
								\cdots \circ \big( W_{N-1}^T\big(\sigma'_N(W_{N-1} \cdots \sigma_1(W_0X)\cdots) )	 \cdots \nonumber \\
								\circ \big(W_N^T(W_N\sigma_N(W_{N-1} \cdots \sigma_1(W_0X)\cdots) - Y)\big)	\cdots\big)\big)\big)\bigg] \nonumber \\
								\sigma_j(W_j \cdots \sigma_1(W_0X))^T.
\end{align}
Stationary points of the loss function are characterized by two main cases:
\begin{enumerate}
\item The misfit is exactly zero. If such points are possible, then these points correspond to local minima

\item For each block the following orthogonality condition holds:
\begin{align} \label{deep_orth_requirement}
&\bigg[\sigma'_{j+1}(W_j\sigma_j \cdots \sigma_1(W_0X)\cdots) \circ \nonumber \\
&\big(W_{j+1}^T\big(\sigma_{j+2}'(W_{j+1}\sigma_j\cdots \sigma_1(W_0X)\cdots) \circ \nonumber \\
&\cdots \circ \big( W_{N-1}^T\big(\sigma'_N(W_{N-1} \cdots \sigma_1(W_0X)\cdots) )	 \cdots \nonumber \\
&\circ \big(W_N^T(W_N\sigma_N(W_{N-1} \cdots \sigma_1(W_0X)\cdots) - Y)\big)	\cdots\big)\big)\big)\bigg] \nonumber \\
& \perp \sigma_j(W_j \cdots \sigma_1(W_0X))^T
\end{align}
\end{enumerate}
\end{theorem}
This result follows from Lemma \ref{deep_nn_grad_lemma}. There are many different conditions on the weights and activation functions that will satisfy the orthogonality requirement in Equation \eqref{deep_orth_requirement}. One specific example is analogous to the condition in Corollary \ref{crit_pt_corollary}.

\begin{corollary}
Any weights $[W_0,\dots,W_{N-1}]$ such that 
\begin{align} 
	&\sigma_N(W_{N-1}\sigma_{N - 1}( \cdots \sigma_1(W_0X)\cdots )) = 0 \label{deep_nn_zero}\\
	&\sigma'_N(W_{N-1}\sigma_{N - 1}( \cdots \sigma_1(W_0X)\cdots )) = 0 \label{deep_deriv_nn_zero}
\end{align}
correspond to a first order stationary point for any $W_N$. 
\end{corollary}

This is the case since the term that is zero in Equation \eqref{deep_nn_zero} shows up in the $W_N$ block of the gradient, and the term that is zero in Equation \eqref{deep_deriv_nn_zero} shows up in every other block of the gradient via an Hadamard product due to the chain rule.

Analysis similar to that in Section \ref{shallow_section} can be
carried out to establish conditions for Hessian rank deficiency at
zero misfit minima and corresponding ill-posedness of the training
problem in a neighborhood, as well as analysis that may establish
conditions for saddle points and spurious local minima. Due to limited
space we do not pursue such analyses, but expect similar
results. Specifically the last activation function and its derivatives
seem to be critical in understanding the characteristics of stationary
points, both their existence and Hessian rank deficiency. If the
successive layer mappings prior to the last layer map
$W_{N-1}\sigma_{N-1}(\dots \sigma_1(W_0X))$ into the zero set of the
last activation and its derivatives then we believe spurious local
minima are possible.

\section{Conclusion}

For dense nonlinear neural networks, we have derived expressions
characterizing the nullspace of the Hessian in the vicinity of global
minima.
% which in turn characterize the ill-posedness of the training
% problem.
These can be used to design  regularization operators that target the
specific nature of ill-posedness of the training problem. 
When a candidate stationary point is a strict saddle,
% it is expectedthat most
appropriately-designed optimization algorithms will escape it
eventually (how fast they escape will depend on how negative the most
negative eigenvalue of the Hessian is). The analysis in this paper
shows that when the gradient is small, it can be due to an accurate
approximation of the mapping $X \mapsto Y$, or it can be due to the
orthogonality condition, Equation \eqref{not_global_minimum}. Spurious
local minima can be identified easily, since $\|Y -
W_1\sigma(W_0X)\|_F$ will be far from zero. Whether or not such points
are strict saddles or local minima is harder to know specifically
since this can depend on many different factors, such as the zeros of
the activation function and its derivatives. Such points can be
escaped quickly using Gaussian random noise
\cite{JinChiGeEtAl2017}. When in the vicinity of a strict saddle point
with a negative curvature direction that is large relative to other
eigenvalues of the Hessian, randomized methods can be used to identify
negative curvature directions and escape the saddle point at a cost
of a small number of neural network evaluations
\cite{OLearyRoseberryAlgerGhattas2019}.

% \section*{Acknowledgements}

% \textbf{Do not} include acknowledgements in the initial version of
% the paper submitted for blind review.

% If a paper is accepted, the final camera-ready version can (and
% probably should) include acknowledgements. In this case, please
% place such acknowledgements in an unnumbered section at the
% end of the paper. Typically, this will include thanks to reviewers
% who gave useful comments, to colleagues who contributed to the ideas,
% and to funding agencies and corporate sponsors that provided financial


\appendix

\section{Shallow Dense Neural Network Derivations} \label{shallow_appendix}

\subsection{Derivation of gradient} \label{derivation_of_grad}
Derivatives are taken in vectorized form. In order to simplify notation we use the following:
\begin{align}
	F(W_1,W_0) &= \frac{1}{2}\text{misfit}^T\text{misfit} \nonumber \\
	\text{misfit} &= \text{vec}(Y-W_1\sigma(W_0X)).
\end{align}

In numerator layout partial differentials with respect to a vectorized matrix $X$ are as follows:
\begin{equation} \label{numerator_layout_misfit_product}
	\partial_{\text{vec}(X)}\bigg(\frac{1}{2}\text{misfit}^T\text{misfit}\bigg) = (\partial_{\text{vec}(X)}\text{misfit})^T\text{misfit}.
\end{equation}
First we have a Lemma about the derivative of the activation function with respect to the encoder weight matrix.

\begin{lemma}\label{diagvec_derivative_lemma}
Suppose $W_0 \in \mathbb{R}^{r \times n}, X \in \mathbb{R}^{n \times d}$ and $\sigma$ is applied elementwise to the matrix $W_0X$, then
\begin{equation}
	\partial_{\text{vec}(W_0)}\text{vec}(\sigma(W_0X)) = \text{diagvec}(\sigma'(W_0X))[X^T\otimes I_r].
\end{equation}
\end{lemma}

\begin{proof}
We use the limit definition of the derivative to derive this result. Let $h\in \mathbb{R}^{r\times n}$ be arbitrary. In the limit as $h \rightarrow 0 \in \mathbb{R}^{r \times n}$ we have the following:
\begin{equation}
	\text{vec}(\sigma((W_0+h)X) - \sigma(W_0X)) = \partial_{\text{vec}(W_0)}(\text{vec}(\sigma(W_0X)))\text{vec}(h)
\end{equation}
Expanding this term and noting that $\text{vec}(\sigma(W_0X)) = \sigma(\text{vec}(W_0X))$, as well as $\text{vec}(A)\circ \text{vec}(B) = \text{diagvec}(A)\text{vec}(B)$, we have:
\begin{align}
&\sigma(\text{vec}((W_0+h)X)) - \sigma(\text{vec}(W_0X)) \nonumber \\
= &\sigma(\text{vec}(W_0X - hX)) - \sigma(\text{vec}(W_0X)) \nonumber \\
= & \text{vec}(\sigma'(W_0X)) \circ \text{vec}(hX) \nonumber \\
= & \text{diagvec}(\sigma'(W_0X)) \text{vec}(hX) \nonumber \\
= & \text{diagvec}(\sigma'(W_0X))[X^T \otimes I_r]\text{vec}(h).
\end{align}
The result follows.
\end{proof}
Now we can derive the gradients of the objective function $F(W_1,W_0)$.

\begin{lemma} \label{derivation_of_grad_lemma}
The gradients of the objective function are given by
\begin{align}
	\nabla_{W_1}F(W_1,W_0) &= (W_1\sigma(W_0X) - Y)\sigma(W_0X)^T \\
	\nabla_{W_0}F(W_1,W_0) &= [\sigma'(W_0X)\circ (W_1^T(W_1\sigma(W_0X) - Y))]X^T.
\end{align}
\end{lemma}
\begin{proof}
We derive in vectorized differential form, from which the matrix form derivatives can be extracted. First for the derivative with respect to $D$ we can derive via the matrix partial differential only with respect to $D$:
\begin{align}
\partial \text{misfit} &= - \partial\text{vec}(W_1\sigma(W_0X)) \nonumber \\
										&= -[\sigma(W_0X)^T \otimes I_n] \partial \text{vec}(W_1) 
\end{align}
Thus it follows that 
\begin{equation} 
	\partial_{\text{vec}(W_1)}\text{misfit} = - [\sigma(W_0X)^T \otimes I_n] \label{misfit_w1}
\end{equation}
The $\text{vec}(W_1)$ partial derivative is then: 
\begin{align}
& (\partial_{\text{vec}(W_1)}\text{misfit})^T\text{misfit} \nonumber \\
= &[\sigma(W_0X)^T \otimes I_n]\text{vec}(W_1\sigma(W_0X) - Y) \nonumber \\
= &\text{vec}((W_1\sigma(W_0X) - Y)\sigma(W_0X)^T).
\end{align}
We have then that the matrix form partial derivative with respect to $W_1$ is:
\begin{equation}
	\nabla_{W_1} F(W_1,W_0) = W_1\sigma(W_0X) - Y)\sigma(W_0X)^T.
\end{equation}
For the partial derivative with respect to $W_0$, again we start with the vectorized differential form.
\begin{align}
	\partial_{\text{vec}(W_0)} \text{misfit} &= -\partial_{\text{vec}(W_0)} \text{vec}(W_1\sigma(W_0X)) \nonumber \\
						&= -[I_d \otimes W_1]\partial_{\text{vec}(W_0)} \sigma(\text{vec}(W_0X))
\end{align}
Applying Lemma \ref{diagvec_derivative_lemma} we have:
\begin{equation}\label{misfit_w0}
	\partial_{\text{vec}(W_0)} \text{misfit} = -[I_d \otimes W_1]\text{diagvec}(\sigma'(W_0X))[X^T \otimes I_r].
\end{equation}
The $\text{vec}(W_0)$ partial derivative is then:
\begin{align}
	&(\partial_{\text{vec}(W_0)}\text{misfit})^T \text{misfit} \nonumber \\
 = & [X\otimes I_r]\text{diagvec}(\sigma'(W_0X))[I_d \otimes W_1^T]\text{vec}(W_1\sigma(W_0X) - Y) \nonumber \\
 = & [X\otimes I_r]\text{diagvec}(\sigma'(W_0X))\text{vec}(W_1^T(W_1\sigma(W_0X) - Y)) \nonumber \\
 = & [X \otimes I_r]\text{vec}(\sigma'(W_0X)\circ (W_1^T(W_1\sigma(W_0X) - Y))) \nonumber \\
 = & \text{vec}([\sigma'(W_0X)\circ (W_1^T(W_1\sigma(W_0X) - Y))]X^T),
\end{align}
We have then that the matrix form partial derivative with respect to $W_0$ is:
\begin{equation}
	\nabla_{W_0} F(W_1,W_0) = [\sigma'(W_0X)\circ (W_1^T(W_1\sigma(W_0X) - Y))]X^T.
\end{equation}
\end{proof}

\subsection{Derivation of Hessian}
We now derive the four blocks of the Hessian matrix. I will proceed again by deriving partial differentials in vectorized form. In numerator layout we have
\begin{align}
	&\partial_{\text{vec}(Y)}\partial_{\text{vec}(X)}\bigg(\frac{1}{2}\text{misfit}^T\text{misfit}\bigg) = \nonumber \\
	& (\partial_{\text{vec}(X)}\partial_{\text{vec}(Y)}\text{misfit})^T\text{misfit} + (\partial_{\text{vec}(X)}\text{misfit})^T(\partial_{\text{vec}(Y)}\text{misfit})
\end{align}
The term involving only first partial derivatives of the misfit is the Gauss Newton portion which are already derived in section \ref{derivation_of_grad}. 

\begin{lemma}{Gauss-Newton portions} \label{gn_hessian_derivation}
\begin{align}
&(\partial_{\text{vec}(W_1)}\text{misfit})^T(\partial_{\text{vec}(W_1)}\text{misfit}) \nonumber\\
 = &[\sigma(W_0X)\sigma(W_0X)^T \otimes I_n] \\
&(\partial_{\text{vec}(W_0)}\text{misfit})^T(\partial_{\text{vec}(W_1)}\text{misfit}) \nonumber\\
 = &[X \otimes I_r]\text{diagvec}(\sigma'(W_0X))[\sigma(W_0X)^T \otimes W_1^T] \\
 &(\partial_{\text{vec}(W_1)}\text{misfit})^T(\partial_{\text{vec}(W_0)}\text{misfit}) \nonumber \\
 = &[\sigma(W_0X) \otimes W_1 ]\text{diagvec}(\sigma'(W_0X))[X^T \otimes I_r]\\
 &(\partial_{\text{vec}(W_0)}\text{misfit})^T(\partial_{\text{vec}(W_0)}\text{misfit}) \nonumber \\
 = &[X \otimes I_r]\text{diagvec}(\sigma'(W_0X))[I_d \otimes W_1^TW_1] \cdots \nonumber \\
 & \qquad \cdots\text{diagvec}(\sigma'(W_0X))[X^T \otimes I_r]
\end{align}
\end{lemma}

\begin{proof}
This result follows from equations \eqref{misfit_w1} and \eqref{misfit_w0}.
\end{proof}

We proceed by deriving the terms involving second partial derivatives of the misfit, by deriving their action on an arbitrary vector $Z \in \mathbb{R}^{n \times d}$. The matrix $K^{(r,d)} \in \mathbb{R}^{dr \times rd}$ is
the commutation (perfect shuffle) matrix satisfying the equality $K^{(r,d)}\text{vec}(V) = \text{vec}(V)^T$ for $V \in \mathbb{R}^{r\times d}$.
\begin{lemma}{Non Gauss-Newton portions} \label{ngn_hessian_derivation}
\begin{align}
&(\partial_{\text{vec}(W_1)}\partial_{\text{vec}(W_1)}\text{misfit})^T \text{misfit} = 0  \label{w1w1_block_ngn}\\
&(\partial_{\text{vec}(W_0)}\partial_{\text{vec}(W_1)}\text{misfit})^T\text{misfit}  \nonumber \\
= &[I_r \otimes (W_1\sigma(W_0X) - Y)]K^{(r,d)}\text{dvec}(\sigma'(W_0X))[X^T \otimes I_r] \label{w0w1_block_ngn} \\
&(\partial_{\text{vec}(W_1)}\partial_{\text{vec}(W_0)}\text{misfit})^T\text{misfit}  \nonumber \\
= &[X \otimes I_r] \text{dec}(\sigma'(W_0X))[(W_1\sigma(W_0X) - Y)^T \otimes I_r]K^{(n,r)} \label{w1w0_block_ngn} \\
&(\partial_{\text{vec}(W_0)}\partial_{\text{vec}(W_0)}\text{misfit})^T\text{misfit}  \nonumber \\
= &[X \otimes I_r]\text{dvec}\big([W_1^T(W_1\sigma(W_0X) - Y)] \nonumber \\
& \qquad \qquad \qquad \circ \sigma''(W_0X)\big)[X^T\otimes I_r] \label{w0w0_block_ngn}
\end{align}
\end{lemma}

\begin{proof}
Equation \eqref{w1w1_block_ngn} follows from the fact that $W_1$ shows up linearly in the misfit. For the $W_0-W_1$ block we have:
\begin{align}
	(\partial_{\text{vec}(W_1)}\text{misfit})^T\text{vec}(Z) &= -[\sigma(W_0X) \otimes I_n]\text{vec}(Z) \nonumber \\
	&= - \text{vec}(Z \sigma(W_0X)^T) \nonumber \\
	&= - [I_r \otimes Z]\text{vec}(\sigma(W_0X)^T) \nonumber \\
	&= - [I_r \otimes Z]K^{(r,d)}\text{vec}(\sigma(W_0X)).
\end{align}
Equation \eqref{w0w1_block_ngn} follows from taking a partial differential with respect to the vectorization of $W_0$, applying Lemma \ref{diagvec_derivative_lemma}, and substituting the misfit for $Z$. For the $W_1-W_0$ block we have:
\begin{align}
&(\partial_{\text{vec}(W_0)}\text{misfit})^T\text{vec}(Z) \nonumber \\
= &-[X\otimes I_r]\text{diagvec}(\sigma'(W_0X))[I_d \otimes W_1^T]\text{vec}(Z) \nonumber \\
= &- [X\otimes I_r]\text{diagvec}(\sigma'(W_0X))\text{vec}(W_1^TZ) \nonumber \\
= &- [X\otimes I_r]\text{diagvec}(\sigma'(W_0X))[Z^T \otimes I_r]K^{(n,r)}\text{vec}(W_1). 
\end{align}
Equation \eqref{w1w0_block_ngn} follows from taking a partial differential with respect to the vectorization of $W_0$ and subsituting the misfit for $Z$. Lastly for the $W_0-W_0$ block we have:
\begin{align}
&(\partial_{\text{vec}(W_0)}\text{misfit})^T\text{vec}(Z) \nonumber \\
= &- [X\otimes I_r]\text{diagvec}(\sigma'(W_0X))\text{vec}(W_1^TZ) \nonumber \\
= &- [X\otimes I_r]\text{vec}(\sigma'(W_0X)) \circ \text{vec}(W_1^TZ) \nonumber \\
= &- [X\otimes I_r]\text{vec}(W_1^TZ) \circ \text{vec}(\sigma'(W_0X))\nonumber \\
= &- [X\otimes I_r]\text{diagvec}(W_1^TZ) \text{vec}(\sigma'(W_0X)).
\end{align}
Equation \eqref{w0w0_block_ngn} follows from taking a partial differential with respect to the vectorization of $W_0$, applying Lemma \ref{diagvec_derivative_lemma}, and substituting the misfit in for $Z$.
\end{proof}

\section{Deep Dense Neural Network Gradient Derivation}
The least squares loss function may be stated as:
\begin{align}
	&F(W_0,W_1,\dots,W_N) = \frac{1}{2}\text{misfit}^T\text{misfit} \nonumber \\
	&\text{misfit} = \text{vec}(Y - W_N\sigma_N(W_{N-1}\sigma_{N_1}( \cdots \sigma_1(W_0X)\cdots ))).
\end{align}
Numerator layout partial differentials are the same as in Equation \eqref{numerator_layout_misfit_product}. The partial derivatives require repeated application of the chain rule and Lemma \ref{diagvec_derivative_lemma}, which can be stated:
\begin{align} \label{deep_chain_rule}
	&\partial_{\text{vec}(W_j)}\text{vec}(\sigma_{j+2}(W_{j+1}\sigma_{j+1}(W_j\sigma_j(\cdots)))) \nonumber \\
	= &\text{dvec}(\sigma'_{j+2}(W_{j+1}\sigma_{j+1}(W_j\sigma_j(\cdots)))) \cdots \nonumber \\
	&[I_d \otimes W_{j+1}]\partial_{\text{vec}(W_j)}\text{vec}(\sigma_{j+1}(W_j\sigma_j(\cdots)))
\end{align}

\begin{lemma}{Deep neural network gradients} \label{deep_nn_grad_lemma}
\begin{align}
\nabla_{W_j}F(\mathbf{W}) = \bigg[\sigma'_{j+1}(W_j\sigma_j \cdots \sigma_1(W_0X)\cdots) \circ \nonumber \\
								\big(W_{j+1}^T\big(\sigma_{j+2}'(W_{j+1}\sigma_j\cdots \sigma_1(W_0X)\cdots) \circ \nonumber \\
								\cdots \circ \big( W_{N-1}^T\big(\sigma'_N(W_{N-1} \cdots \sigma_1(W_0X)\cdots) )	 \cdots \nonumber \\
								\circ \big(W_N^T(W_N\sigma_N(W_{N-1} \cdots \sigma_1(W_0X)\cdots) - Y)\big)	\cdots\big)\big)\big)\bigg] \nonumber \\
								\sigma_j(W_j \cdots \sigma_1(W_0X))^T
\end{align}
\end{lemma}
\begin{proof}
By iterative application of the chain rule (Equation \eqref{deep_chain_rule}) we can derive the following
\begin{align}
&\partial_{\text{vec}(W_j)} \text{misfit} = \nonumber \\
- &[I_d \otimes W_N]\text{dvec}(\sigma_N'(W_N \cdots \sigma_1(W_0X) \cdots))\cdots  \nonumber \\
 &[I_d \otimes W_{j+1}]\text{dvec}(\sigma'_{j+1}(W_j \cdots \sigma_1(W_0X) \cdots )) \nonumber \\
 &[\sigma_j(W_j \cdots \sigma_1(W_0X) \cdots ) \otimes I_{r_j}].
\end{align}
The result then follows from Equation \eqref{numerator_layout_misfit_product} and properties of Kronecker and Hadamard products that are used in Appendix \ref{shallow_appendix}.
\end{proof}

\bibliography{arxiv_illposed.bib}
\bibliographystyle{icml2020}


%%%%%%%%%%%%%%%%%%%%%%%%%%%%%%%%%%%%%%%%%%%%%%%%%%%%%%%%%%%%%%%%%%%%%%%%%%%%%%%
%%%%%%%%%%%%%%%%%%%%%%%%%%%%%%%%%%%%%%%%%%%%%%%%%%%%%%%%%%%%%%%%%%%%%%%%%%%%%%%

\end{document}